\documentclass[11pt]{article}
\usepackage{amsmath,amssymb}

\newtheorem{propo}{{\bf Proposition}}[section]
\newtheorem{coro}[propo]{{\bf Corollary}}
\newtheorem{lemma}[propo]{{\bf Lemma}} \newtheorem{theor}[propo]{{\bf
Theorem}} 
\newtheorem{defn}{\bf Definition}
\newenvironment{proof}{{\bf Proof.}}{$\Box$}

\begin{document}
\vspace*{1.0in}

\begin{center} Nilpotency, Solvability and Frattini Theory for Poisson algebras

\end{center}
\bigskip

\centerline {David A. Towers} \centerline {Department of
Mathematics, Lancaster University} \centerline {Lancaster LA1 4YF,
England}
\bigskip

\begin{abstract} This paper shows that a Poisson algebra is nilpotent if and only if it is both associative and Lie nilpotent and examines various properties of the nilradical and the solvable radical. It introduces a basic Frattini theory for dialgebras and then investigates a more detailed theory for Poisoon algebras.
\end{abstract}


\section{Introduction}
A dialgebra $(\mathcal{A}, \cdot, [\cdot, \cdot])$ is a vector space $\mathcal{A}$ endowed with two multiplications $\cdot$ and $[\cdot, \cdot]$ that are not necessarily associative. Some popular classes of dialgebras are Lie-Yamaguti algebras, Gerstenhaber algebras, Nambu-Poisson algebras, Novikov-Poisson algebras, Gelfand-Dorfman algebras, and many others, including Poisson algebras, which are the main objects of study of this paper.

\begin{defn}\rm
A {Poisson algebra} $\mathcal{P}$ is a dialgebra $(\mathcal{P}, \cdot, [\cdot, \cdot])$ such that $\mathcal{P}_A:=(\mathcal{P}, \cdot)$ is an associative commutative algebra, $\mathcal{P}_L:=(\mathcal{P}, [\cdot, \cdot])$ is a Lie algebra and they satisfy the compatibility identity for $x,y,z\in \mathcal{P}$ given by
\begin{equation*}
[x\cdot y, z]=[x,z]\cdot y + x\cdot[y,z] \quad \quad \textrm{(Leibniz rule).}
\end{equation*}
\end{defn}

The significance of Poisson algebras lies in their ability to model the dynamics of classical mechanical systems. They are instrumental in the formulation of Hamiltonian mechanics, where the Poisson bracket defines the evolution of observables over time. This makes Poisson algebras indispensable in the study of integrable systems, symplectic geometry, and the theory of deformations.

Beyond classical mechanics, Poisson algebras find applications in quantum mechanics through the process of quantization. They serve as a bridge between classical and quantum theories, facilitating the transition from Poisson brackets to commutators in quantum algebras. This connection is pivotal in understanding the underlying algebraic structures of quantum field theory and string theory.

Moreover, Poisson algebras are utilized in various branches of mathematics, including algebraic geometry and representation theory. They provide tools for exploring the geometry of algebraic varieties and the symmetries of differential equations. In computational mathematics, Poisson algebras are employed in algorithms for solving polynomial equations and in the analysis of dynamical systems.  For more detailed information on  some of these studies see \cite{bv}, \cite{cfm}, \cite{drin},\cite{huch}, \cite{kont},  \cite{rs} and the references contained therein.

There have also been studies of the purely algebraic structure of Poisson algebras: see, for instance,  \cite{gr}, \cite{mr}, \cite{ont}, \cite{su}. This paper is a further contribution to this latter study. 

Every algebra in this paper is finite-dimensional.  All the vector spaces, algebras and linear maps are considered over an arbitrary field $\mathbb{F}$, unless we say otherwise. The notation `$\subseteq$' will denote inclusion, whilst `$\subset$' indicates strict inclusion. We will use `$\dot{+}$' to denote a direct sum of the underlying vector space structure and `$\oplus$' for an algebra direct sum.  Let us recall some definitions.

\begin{defn}
A {subalgebra}  of a dialgebra $\mathcal{A}$ is a linear subspace $A$ closed by both multiplications, that is $A \cdot A + [A,A]    \subset A$.
A subalgebra $I$ of $%
\mathcal{A}$ is  an {ideal } if $I \cdot \mathcal{A} + \mathcal{A}\cdot I +[I,\mathcal{A}] + [\mathcal{A}, I]    \subset I$. A zero subalgebra (or ideal) is a subalgebra (or ideal)  $A$ such that $A \cdot A + [A,A]    = 0$. 
\end{defn}

Given two dialgebras  $\mathcal{A}_{1}$ and $\mathcal{A}
_{2}$, a {homomorphism} is a linear map $\phi :\mathcal{A}_{1}\longrightarrow \mathcal{A}_{2}$ satisfying, for all $%
x,y  \in\mathcal{A}_{1}$, that $\phi \left(  x \cdot_{\mathcal{A}_{1}} y \right) =  \phi \left(
x\right)\cdot_{\mathcal{A}_{2}}   \phi \left( y\right)  $  and
$\phi \left( [x,y]_{\mathcal{A}_{1}} \right) =[ \phi \left(
x\right) ,\phi \left( y\right) ]_{\mathcal{A}_{2}} $. If $\phi :\mathcal{A}_{1}\longrightarrow
\mathcal{A}_{2}$ is a homomorphism of dialgebras, then $\ker \phi $ is an ideal
of $\mathcal{A}_{1}$ and $\phi \left( \mathcal{A}_{1}\right) \cong \mathcal{A
}_{1}/\ker \phi $.

 If $I$ is an ideal of a dialgebra $\mathcal{A}$,  the quotient vector space $\overline{\mathcal{A}}=\mathcal{A}/I$  together with the  operations:
\begin{equation*}
(x+I) \cdot (y+I):= x \cdot y
+I,
\end{equation*}
\begin{equation*}
[ x+I,y+I ]:=[ x ,y  ]
+I,
\end{equation*}
for all $x,y \in \mathcal{A}$, is called the quotient dialgebra.
The mapping $\mathcal{A} \longrightarrow  \mathcal{A}/I$ defined by  $x \longmapsto  x + I$ is
a homomorphism  of dialgebras with kernel $I$.

{Given a dialgebra $(\mathcal{A}, \cdot, [\cdot, \cdot])$, recall the derived
sequence of subspaces. For $n \geq 0$,  define
$$\mathcal{A}^{(0)}:=\mathcal{A} \quad \quad \quad \mathcal{A}^{(n+1)} =  \mathcal{A}^{(n)}\cdot\mathcal{A}^{(n)} + [\mathcal{A}^{(n)}, \mathcal{A}^{(n)}].$$

\begin{defn} 
\rm
A dialgebra $(\mathcal{A}, \cdot, [\cdot, \cdot])$ is {solvable} if there exist $m\geq 0$ such that $\mathcal{A}^{(m) }=0$. We will call $\mathcal{A}$ supersolvable if there is a flag $$0=\mathcal{A}_0\subset \mathcal{A}_1\subset \ldots \subset \mathcal{A}_n=\mathcal{A},$$ where $\mathcal{A}_i$ is an $i$-dimensional ideal of $\mathcal{A}$ for $1\leq i\leq n$.
\end{defn}

Moreover, the lower central series is the sequence
$$\mathcal{A}^1:=\mathcal{A} \quad \quad \quad \mathcal{A}^{n+1} = \sum_{i=1}^{n}( \mathcal{A}^{i}\cdot\mathcal{A}^{n+1-i} + [\mathcal{A}^{i}, \mathcal{A}^{n+1-i}]).$$

\begin{defn} 
\rm
A dialgebra $(\mathcal{A}, \cdot, [\cdot, \cdot])$ is {nilpotent} if there exists $m\geq 0$ such that $\mathcal{A}^m=0$. It is a zero algebra if $\mathcal{A}^2=0$. 
\end{defn}

\begin{defn}
\rm
If $B$ is a subalgebra of a dialgebra $\mathcal{A}$, the annihilator of $B$ in $\mathcal{A}$, $Ann_\mathcal{A}(B)$, is defined by $$Ann_\mathcal{A}(B)=\{x\in \mathcal{A} \mid x\cdot B+B\cdot x+[x,B]+[B,x]=0\}$$. If $B$ is an ideal of $\mathcal{A}$ it is easy to check that $Ann_\mathcal{A}(B)$ is also an ideal of $\mathcal{A}$. The centre of $\mathcal{A}$, $Z(\mathcal{A})$, is equal to $Ann_\mathcal{A}(\mathcal{A})$.
\end{defn}

Note that if a dialgebra is solvable (resp. nilpotent), then each of the multiplications is solvable (resp. nilpotent). Also, if $\mathcal{A}$ is a Poisson algebra, then we have 
$
\mathcal{A}^{n+1}=  \mathcal{A}^n  \cdot  \mathcal{A}   +
[ \mathcal{A}^n ,  \mathcal{A}  ]$. 

\begin{defn} For a Poisson algebra $\mathcal{P}$ we define the series $$\mathcal{P}_A^{(0)}=\mathcal{P}, \quad \mathcal{P}_A^{(n+1}=\mathcal{P}_A^{(n)}\cdot \mathcal{P}_A^{(n)}\quad \hbox{and}\quad \mathcal{P}_L^{(0)}=\mathcal{P} \quad \mathcal{P}_L^{(n+1)}=[\mathcal{P}_L^{(n)},\mathcal{P}_L^{(n)}].$$ Also,
$$\mathcal{P}_A^1=\mathcal{P}, \quad \mathcal{P}_A^{n+1}=\mathcal{P}_A^n\cdot \mathcal{P}\quad \hbox{and} \quad \mathcal{P}_L^1=\mathcal{P}, \quad \mathcal{P}_L^{n+1}=[\mathcal{P}_L^n,\mathcal{P}].$$ We say that $\mathcal{P}$ is associative (respectively, Lie) solvable if there exists $m > 0$ such that $\mathcal{P}_A^{(m)}=0$ (respectively, $\mathcal{P}_L^{(m)}=0$) . Similarly, we can define associative and Lie nilpotency.
\end{defn}

{\bf Note} When speaking of an `ideal' of a Poisson algebra we will mean a Poisson ideal. We should also point out that our definitions and notation are not used universally in the literature.
\medskip

In section two we collect together some preliminary results. First we show that a Poisson algebra is nilpotent if and only if it is both associastive and Lie nilpotent, thereby establishing that the nilradical exists. We then show that if $R$ is the solvable radical of a Poisoon algebra, then $R_A^2$ is nilpotent, and if the field has characteristic zero, $R^2$ is nilpotent. Finally, we introduce Engel subalgebras and the subalgebras  $S_{\mathcal{P}}(a,\mathbb{F})$ of a Poisson algebra.
\par

In section 3 we define the Frattini subalgebra and ideal of a dialgebra, and collect together some basic results concerning them which mirror the corresponding concepts for an algebra.
\par

The final section contains the main results concerning the Frattini ideal of a Poisson algebra. the first three mirror the corresponding results for other classes of algebra. We go on to study Poisson algebras with trivial Frattini ideal, showing, in particular that for such a solvable algebra the associative multiplication is trivial. We finish by classifying all Poisson algebras in which every maximal subalgebra is an ideal.

\section{Preliminary results}
\begin{lemma}\label{1} Let $B,C$ be subalgebras of the Poisson algebra $\mathcal{P}$. Then $$[B_A^n,C]\subseteq B_A^{n-1}\cdot [B,C] \hbox{ for all } n\geq1.$$
\end{lemma}
\begin{proof} Suppose the result holds for $n=k$. Then $$[B_A^{k+1},C]=[B_A^k\cdot B,C]\subseteq B_A^k\cdot[B,C]+B\cdot [B_A^k,C]\subseteq B_A^k\cdot [B,C].$$
\end{proof}

\begin{lemma}\label{2} If $B, C$ are ideals of the Poisson algebra $\mathcal{P}$. then so is $B\cdot C$.
\end{lemma}
\begin{proof} It is clear that $B\cdot C \cdot \mathcal{P}\subseteq B\cdot C$. Also, $$[B\cdot C,\mathcal{P}]\subseteq B\cdot [C,\mathcal{P}]+C\cdot [B,{\mathcal P}]\subseteq B\cdot C,$$ whence the result.
\end{proof}

\begin{lemma}\label{ann} Let $B$ be a minimal ideal of a Poisson algebra $\mathcal{P}$, let $N$ be an associative and Lie nilpotent ideal of $\mathcal{P}$, and suppose that $B\subseteq N$. Then $B\subseteq Ann_\mathcal{P}(N)$.
\end{lemma}
\begin{proof} We have that  $B\cdot N$ is an ideal of $\mathcal{P}$, by Lemma \ref{2}. From the minimality we have that $B\cdot N=B$ or $B\cdot N=0$. The former is impossible, by the associative nilpotency of $N$, so $B\cdot N=0$.. But then $$[B,N]\cdot \mathcal{P} +[[B,N],\mathcal{P}]\subseteq [B\cdot N,\mathcal{P}]+B\cdot[N,\mathcal{P}]+[B,N]\subseteq [B,N],$$ so $[B,N]$ is an ideal of $\mathcal{P}$. A similar argument using the Lie nilpotency of $N$ shows that $[B,N]=0$. Hence  $B\subseteq Ann_\mathcal{P}(N)$.
\end{proof}

\begin{propo}\label{nil} A Poisson algebra $\mathcal{P}$ is nilpotent if and only if it is both associative and Lie nilpotent.
\end{propo}
\begin{proof} Suppose first that $\mathcal{P}$ is nilpotent. Then $\mathcal{P}_A^n, \mathcal{P}_L^n\subseteq \mathcal{P}^n$ for all $n\geq 1$, so $\mathcal{P}$ is both associative and Lie nilpotent. 
\par

Suppose next that $\mathcal{P}$ is both associative and Lie nilpotent. We show that $\mathcal{P}$ is nilpotent by induction on its dimension. If $\dim \mathcal{P}=1$, it is a zero algebra and so is nilpotent. So suppose that the result holds when $\dim \mathcal{P}\leq k$ and let $\mathcal{P}$ have dimension $k+1$.
\par

Let $A$ be a minimal ideal of $\mathcal{P}$. Then $\dim \mathcal{P}/A\leq k$, so $\mathcal{P}^m\subseteq A$ for some $m$. But now $$ \mathcal{P}^{m+1}=\mathcal{P}^m\cdot \mathcal{P}+[\mathcal{P}^m,\mathcal{P}]\subseteq A\cdot \mathcal{P}+[A,\mathcal{P}]=0,$$ by Lemma \ref{ann}. Hence $\mathcal{P}$ is nilpotent.
\end{proof}

\begin{defn} If $B$ and $C$ are ideals of a Poisson algebra $\mathcal{P}$ which are both associative and Lie nilpotent, then $B+C$ is another such ideal. There is thus a maximal associative and Lie nilpotent ideal $N(\mathcal{P})$ that we will call the nilradical of $\mathcal{P}$. It is clear from Proposition \ref{nil} that $N(\mathcal{P})$ is the maximal nilpotent ideal of $\mathcal{P}$. Similarly, there is a maximal solvable ideal of ${\mathcal P}$, $R({\mathcal P})$, which we call the radical of ${\mathcal P}$
\end{defn}

\begin{theor}\label{square} If ${\mathcal P}$ is a Poisson algebra with solvable radical $R$, then $R_A^2\subseteq N({\mathcal P})$.
\end{theor}
\begin{proof} (i) Clearly $R_A^2$ is associative nilpotent. We show that it is Lie nilpotent by induction on $\dim R$. It clearly holds if $\dim R=1$, Suppose it holds whenever $\dim R<k$ ($k\geq 2$) and let $\dim R=k$. The result certainly holds if $R_A^2=0$, so suppose that $R_A^2\neq 0$. Then there is a minimal idea $B$ of ${\mathcal P}$ inside $R_A^2$. Now $R_A^2/B=(R_A/B)^2$ is Lie nilpotent, by the inductive hypothesis.
\par

Clearly $R\cdot B=0$, since $R\cdot B$ is an ideal of ${\mathcal P}$ and $R$ is associative nilpotent. Moreover, $[R_A^2,B]\subseteq R\cdot[R,B]\subseteq R\cdot B=0$. It follows that $R_A^2$ is Lie nilpotent, and thus is nilpotent, by Proposition \ref{nil}.
\end{proof}

\begin{coro}\label{csquare}  Let ${\mathcal P}$ be a Poisson algebra over a field $\mathbb{F}$ of characteristic zero. Then $R^2$ is nilpotent.
\end{coro}
\begin{proof} Since $[R,R]$ and $R_A^2$ are Lie nilpotent Lie ideals of ${\mathcal P}$, $R^2$ is Lie nilpotent. As it is also associative nilpotent, it is nilpotent. The result follows.
\end{proof}

\begin{propo}\label{sup} Let $\mathcal{P}$ be a supersolvable Poisson algebra. Then $\mathcal{P}^2$ is nilpotent.
\end{propo}
\begin{proof} Clearly, $\mathcal{P}^2$ is associative nilpotent. Moreover, $\mathcal{P}_L$ is a supersolvable Lie algebra, so $[\mathcal{P},\mathcal{P}]$ is Lie nilpotent. But $\mathcal{P}_A^2$ is also Lie nilpotent, by Theorem \ref{square}, so $\mathcal{P}^2$ is Lie nilpotent, and thus nilpotent, by Proositionn \ref{nil}.
\end{proof}

\begin{lemma}\label{annrad} Let $\mathcal{P}$ be a Poisson algebra with radical $R$ and nilradical $N$. Then $Ann_R(N)\subseteq N$.
\end{lemma}
\begin{proof} Suppose that $Ann_R(N)\not \subseteq N$. Let $A/N$ be a minimal ideal of $\mathcal{P}/N$ inside $(Ann_R(N)+N)/N$. Then $A/N$ is a zero ideal and $A$ is a nilpotent ideal of $\mathcal{P}$, whence $A\subseteq N$, a contradiction. The result follows.
\end{proof}

\begin{defn} For a Poisson algebra $\mathcal{P}$, we denote by $P_x$ and $Q_x$ the maps in $\textrm{End}(\mathcal{P})$ given by $P_x(y) = x\cdot y$ and $Q_x(y) = [x, y]$, for $x, y \in \mathcal{P}$. For $a\in {\mathcal P}$, put $E_{\mathcal P}^A(a)=\{x\in {\mathcal P}\mid P_a^n(x)=0$ for some $n\}$ and $E_{\mathcal P}^L(a)=\{x\in {\mathcal P}\mid Q_a^n(x)=0$ for some $n\}$.
\end{defn}

\begin{lemma} If ${\mathcal P}$ is a Poisson algebra, then $E_{\mathcal P}^A(a)$ and $E_{\mathcal P}^L(a)$ are subalgebras of ${\mathcal P}$ which we will call Engel subalgebras of ${\mathcal P}$.
\end{lemma}
\begin{proof} Let $x,y\in E_{\mathcal P}^A(a)$, so $ P_a^n(x)=P_a^n(y)=0$ for some $n$. Clearly, $P_a^n(x\cdot y)=x\cdot P_a^n(y)=0$. Now, a straightforward induction proof yields that
\[ P_a^n([x,y])=[P_a^n(x),y]-nP_a^{n-1}(x)\cdot [y,a],
\] so $P_a^{n+1}([x,y])=0$. Hence  $E_{\mathcal P}^A(a)$ is a subalgebra of ${\mathcal P}$.
\par

Let $x,y\in E_{\mathcal P}^L(a)$, so $Q_a^n(x)=Q_a^n(y)=0$ for some $n$. Clearly, $Q_a^r([x,y])=0$ for some $r$, using the Jacobi identity. Now, the Leibniz condition implies that $Q_a$ acts as a derivation on the associative structure of ${\mathcal P}$, so
\[ Q_a^r(x\cdot y)=\sum_{i=0}^r \left( \begin{array}{c} r\\ i \end{array}\right) Q_a^i(x)\cdot Q_a^{r-i}(y).
\] Choosing $r\geq 2n$ gives $Q_a^r(x\cdot y)=0$. It follows that $E_{\mathcal P}^L(a)$ is a subalgebra of ${\mathcal P}$.
\end{proof}

\begin{defn} If $U$ is a subalgebra of a Poisson algebra ${\mathcal P}$, the idealiser (respectively, Lie idealiser, associative idealiser) of $U$ in ${\mathcal P}$ is $I_{\mathcal P}(U)=\{x\in {\mathcal P} \mid [x,u]+x\cdot u\in U \hbox{ for all } u\in U\}$ (respectively $I_{\mathcal P}^L(U)=\{x\in {\mathcal P} \mid [x,u]\in U \hbox{ for all } u\in U\}$, $I_{\mathcal P}^A(U)=\{x\in {\mathcal P} \mid x\cdot u\in U \hbox{ for all } u\in U\}$).
\end{defn}

We will need the following result.

\begin{lemma}\label{norm} If $U$ is a Lie subalgebra of a Poisson algebra ${\mathcal P}$ containing  $E_{\mathcal P}^L(a)$, then  $I_{\mathcal P}^L(U)=U$.
\end{lemma}
\begin{proof} This is a purely Lie algebraic result which is proved as \cite[Theorem 4.4.4.4]{winter}.
\end{proof}

\begin{defn} Following Barnes and Newell in \cite{bn} we put, for $a\in \mathcal{P}$, 
\begin{align*}
&S_{\mathcal{P}}(a,,\mathbb{F}) \\
&= \left\{x\in \mathcal{P}\mid \prod_{i=1}^n(Q_a-\lambda_i1)(x)=0 \hbox{ for some } \lambda_1, \ldots, \lambda_n\in \mathbb{F} \hbox{ and some } n\right\}.
\end{align*} If $\lambda_1, \ldots, \lambda_n$ are the eigenvalues of $Q_a$ in $\mathbb{F}$ and $r_i$ is the multiplicity of $\lambda_i$, then, putting $f(t)=\displaystyle\prod_{i=1}^n(t-\lambda_i)^{r_i}$, $S_{\mathcal{P}}(a,\mathbb{F})=\{x\in \mathcal{P}\mid f(Q_a)(x)=0\}$ and $K_{\mathcal{P}}(a,\mathbb{F})=f(Q_a)(\mathcal{P})$ is a complementary $Q_a$-stable subspace of $\mathcal{P}$.
\end{defn}

\begin{lemma} If $\mathcal{P}$ is a Poisson algebra and $a\in \mathcal{P}$, then $S_{\mathcal{P}}(a,\mathbb{F})$ is a subalgebra of $\mathcal{P}$.
\end{lemma}
\begin{proof} It is a Lie subalgebra of $\mathcal{P}$, by \cite[paragraph 2.3]{bn}. Moreover, if $\displaystyle\prod_i(Q_a-\lambda_i1)(x)=0$ and $\displaystyle\prod_j(Q_a-\mu_j(y)=0$, then $\displaystyle\prod_{i,j}(Q_a-(\lambda_i+\mu_j)1)(x\cdot y)=0$, as in that same paragraph, using the fact that $Q_a$ acts as a derivation on the associative structure of $\mathcal{P}$.
\end{proof}

\section{The Frattini subalgebra of a dialgebra}
\begin{defn} The Frattini subalgebra, $F(\mathcal{A})$, of a dialgebra $\mathcal{A}$ is the intersection of the maximal subalgebras of $\mathcal{A}$; The Frattini ideal, $\phi(\mathcal{A})$, is the largest ideal of $\mathcal{A}$ contained in $F(\mathcal{A})$. As in the other structures in which it has been defined, $F(\mathcal{A})$ is the set of nongenerators of $\mathcal{A}$. If ${\mathcal P}$ is a Poisson algebra, we denote by $F_A({\mathcal P})$ (respectively, $F_L({\mathcal P})$) the intersection of the maximal associative (respectively, Lie) subalgebras of ${\mathcal P}$. Correspondingly, $\phi_A({\mathcal P})$ (respectively, $\phi_L({\mathcal P})$) is the largest associative (respectively, Lie) ideal contained in $F_A({\mathcal P})$ (respectively $F_L({\mathcal P})$).
\end{defn}

The following results are proved in much the same way as the corresponding results in \cite{frat}. Some proofs are given for the convenience of the reader.

\begin{lemma}\label{sub} If $C$ is a subalgebra of the dialgebra $\mathcal{A}$, and $B$ is an ideal of $\mathcal{A}$ contained in $F(C)$, then $B$ is contained in $F(A)$.
\end{lemma}
\begin{proof} Suppose that $B\not\subseteq F(A)$. Then there is a maximal subalgebra $M$ of $\mathcal{A}$ such that $\mathcal{A}=B+M$ and
\[ C=B+M\cap C=F(C)+M\cap C=M\cap C,
\] so $B\subseteq C\subseteq M$, a contradiction. the result follows.
\end{proof}

\begin{lemma}\label{fac} Let $B$ be an ideal of the dialgebra $\mathcal{A}$. Then
\begin{itemize}
\item[(i)] $(F(\mathcal{A})+B)/B\subseteq F(\mathcal{A}/B)$ and $(\phi(\mathcal{A})+B)/B\subseteq \phi(\mathcal{A})/B$;
\item[(ii)] if $B\subseteq F(\mathcal{A})$, $F(\mathcal{A})/B=F(\mathcal{A}/B)$ and $\phi(\mathcal{A})/B=\phi(\mathcal{A}/B)$.
\end{itemize}
\end{lemma}

\begin{lemma}\label{cor} If $R$ is an ideal of the dialgebra  $\mathcal{A}$ and $F(\mathcal{A}/R)=0$ (respectively, $\phi(\mathcal{A}/R)=0$), then $F(\mathcal{A})\subseteq R$ (respectively, $\phi(\mathcal{A})\subseteq R$).
\end{lemma}

\begin{theor}\label{dsum} If $\mathcal{A}$ is a dialgebra and $\mathcal{A}=A_1\oplus \ldots \oplus A_n$, then $\phi(\mathcal{A})=\phi(A_1)\oplus \ldots \oplus \phi(A_n)$.
\end{theor}
\begin{proof} It is easy to see that $F(\mathcal{A})\subseteq F(A_1)+ \ldots +F(A_n)$, so 
\[  \phi(\mathcal{A})\cap A_i\subseteq F(\mathcal{A})\cap A_i\subseteq F(A_i).
\] Moreover, since $\phi(\mathcal{A})\cap A_i$ is an ideal of $A_i$, $\phi(\mathcal{A})\cap A_i\subseteq \phi(A_i)$. Also, $\phi(A_i)$ is an ideal of $\mathcal{A}$ and is contained in $F(A_i)$, so $\phi(A_i)\subseteq F(\mathcal{A})$, by Lemma \ref{sub}. Consequently, $\phi(A_i)\subseteq \phi(\mathcal{A})\cap A_i$, whence $\phi(A_i)=\phi(\mathcal{A})\cap A_i$. Thus $\phi(A_1)\oplus \ldots \oplus \phi(A_n)\subseteq \phi(\mathcal{A})$.
\par

Now suppose that $x\in \phi(\mathcal{A})$, so $x=\sum_{i=1}^nf_i$, where $f_i\in F(A_i)$. Then
\begin{align*}
 x\cdot A_i+A_i\cdot x+[x,A_i]+[A_i,x]&= f_i\cdot A_i+A_i\cdot f_i+[f_i,A_i]+[A_i,f_i] \\
&\subseteq \phi(\mathcal{A})\cap A_i=\phi(A_i).
\end{align*}  It follows that $f_i\in \phi(A_i)$, whence $x\in \phi(A_1)\oplus \ldots \oplus \phi(A_n)$.
\end{proof}

\begin{lemma}\label{min} Let $B$ be an ideal of a dialgebra ${\mathcal A}$, and let $U$ be a subalgebra of ${\mathcal A}$ which is minimal with respect to the property that ${\mathcal A}=B+U$.Then $B\cap U\subseteq \phi(U)$.
\end{lemma}
\begin{proof} Suppose that $B\cap U\not\subseteq \phi(U)$. Then 
\[
U\cdot (B\cap U)+(B\cap U)\cdot U+[U,B\cap U]+[B\cap U,U]\subseteq B\cap U,
\] so $B\cap U$ is an ideal of $U$ and there is a maximal subalgebra $M$ of $U$ such that $B\cap U\not\subseteq M$. It is easy to see that $U=B\cap U+M$, and so
\[ {\mathcal P}=B+(B\cap U+M)=B+M,
\] contradicting the minimality of $U$. The result follows.
\end{proof}

\begin{lemma}\label{splits} Let $B$ be a zero ideal of a dialgebra ${\mathcal A}$ such that $B\cap \phi({\mathcal A})=0$. Then there is a subalgebra $C$ of ${\mathcal A}$ such that ${\mathcal A}=B\dot{+}C$. (We say that ${\mathcal A}$ {\bf splits} over $B$.)
\end{lemma}
\begin{proof} Choose $C$ to be a subalgebra of ${\mathcal A}$ which is minimal with respect to ${\mathcal A}=B+C$. Then, by Lemma \ref{min}, $B\cap C\subseteq \phi(C)$. Now $B\cap C$ is an ideal of $C$ (as in Lemma \ref{min}) and 
\[
B\cdot (B\cap C)+(B\cap C)\cdot B+[B,B\cap C]+[B\cap C,B]\subseteq B^{(1)}=0,
\] so $B\cap C$ is an ideal of ${\mathcal A}$. It folows from Lemma \ref{sub} that $B\cap C\subseteq \phi(\mathcal{A})\cap B=0$ and $\mathcal{A}=B\dot{+}C$.
\end{proof}

\section{Main results}
\begin{defn}  The factor algebra $A/B$ is called a chief factor of $\mathcal{P}$ if $B$ is an ideal of $\mathcal{P}$ and $A/B$ is a minimal ideal of $\mathcal{P}/B$. 
\end{defn}

\begin{theor}\label{factor} Let $B$ be a subideal of a Poisson algebra ${\mathcal P}$, and let $C$ be an ideal of $B$ with $C\subseteq \phi({\mathcal P})$. If $B/C$ is nilpotent or supersolvable, then so is $B$.
\end{theor}
\begin{proof} Let $b\in B$ and $B=B_0\subseteq B_1\subseteq \ldots\subseteq B_r={\mathcal P}$ be a chain of subalgebras of ${\mathcal P}$ with $B_i$ an ideal of $B_{i+1}$ for $0\leq i\leq r-1$. Then $P_b^r({\mathcal P})+Q_b^r({\mathcal P}) \subseteq B$. 
\par

Suppose first that $B/C$ is nilpotent. Since $B/C$ is associative and Lie nilpotent, there exists $s$ such that $P_b^s(B)+Q_b(B)\subseteq C$. Hence $P_b^{r+s}({\mathcal P})+Q_b^{r+s}\subseteq \phi({\mathcal P})$. But $P_b^{r+s}({\mathcal P})+E_{\mathcal P}^A(b)={\mathcal P}$, and $Q_b^{r+s}({\mathcal P})+E_{\mathcal P}^L(b)={\mathcal P}$, by Fitting's Lemma, so $E_{\mathcal P}^A(b)=E_{\mathcal P}^L={\mathcal P}$. It follows that $B$ is both associative and Lie nilpotent, by Engel's theorem and the fact that commutative associative nilalgebras are nilpotent. Hence $B$ is nilpotent, by Proposition \ref{nil}. 
\par

Now assume that $B/C$ is supersolvable. Then $K_{\mathcal{P}}(b,\mathbb{F})\subseteq C$ for all $b\in B$, whence $\mathcal{P}=C+S_{\mathcal{P}}(b,\mathbb{F})$. But $C\subseteq \phi(\mathcal{P})$, so $\mathcal{P}=S_{\mathcal{P}}(b,\mathbb{F})$ for all $b\in B$. Moreover, $(B/C)^2$ is nilpotent, by Proposiition \ref{sup}, so $B^2/B^2\cap                                   C$ is nilpotent, which yields that $B^2$ is nilpotent, by the paragraph above. It follows from \cite[Lemma 2.4]{bn} that $B_L$ is supersolvable. Let $B_i/B_{i+1}$ be a  chief factor for $B$. Then $B$ is associative nilpotent, since $B_A^2\subseteq B^2$ is nilpotent. It follows that $B.B_i\subseteq B_{i+1}$ and hence that $B_i/B_{i+1}$ is a minimal Lie ideal of $B/B_{i+1}$. Thus, $\dim (B_i/B_{i+1}=1$ and $B$ is supersolvable.
\end{proof}

\begin{coro}\label{phi}  If ${\mathcal P}$ is a Poisson algebra, then $\phi({\mathcal P})$ is nilpotent.
\end{coro}

\begin{defn} If $\mathcal{P}$ is a Poisson algebra, the {\bf socle}, denoted $Soc(\mathcal{P})$ (respectively, {\bf zero socle}, denoted $Zsoc(\mathcal{P})$) of ${\mathcal P}$ is the sum of the minimal ideals (respectively, minimal zero ideals) of ${\mathcal P}$. We will say that ${\mathcal P}$ is {\bf $\phi$-free} (respectively $\phi_A$-free, $\phi_L$-free) if $\phi({\mathcal P})=0$ (respectively, $\phi_A({\mathcal P})=0$, $\phi_L({\mathcal P})=0$).
\end{defn}

\begin{theor}\label{split}  Let  ${\mathcal P}$  be a Poisson algebra. Then  ${\mathcal P}$ is $\phi$-free if and only if it splits over its zero socle.
\end{theor}
\begin{proof} If  ${\mathcal P}$ is $\phi$-free, it splits over its zero socle, by Lemma \ref{splits}. So suppose that ${\mathcal P}=Zsoc(A)\dot{+} C$, where $C$ is a subalgebra of ${\mathcal P}$ and $Zsoc(A)=Z_1\oplus\ldots\oplus Z_n$, where $Z_i$ is a minimal zero ideal of ${\mathcal P}$, and that $\phi({\mathcal P})\neq 0$. Then there is a minimal ideal $Z$ of ${\mathcal P}$ contained in $\phi({\mathcal P})$. But $\phi({\mathcal P})$ is nilpotent, by Corollary \ref{phi}, so $Z$ is a zero ideal. Thus $\phi({\mathcal P})\cap Zsoc({\mathcal P})\neq 0$. However, $M_i=(Z_1\oplus\ldots\oplus \hat{Z_i}\oplus\ldots\oplus Z_n)\dot{+}C$, where $\hat{Z_i}$ indicates a term that is missing, is a maximal subalgebra of ${\mathcal P}$ for each $1\leq i\leq n$. Hence
\[ \phi({\mathcal P})\subseteq \cap_{i=1}^n M_i \subseteq C,
\] and $\phi({\mathcal P})\cap Zsoc({\mathcal P}) =0$, a contradiction. It follows that $\phi({\mathcal P})=0$.
\end{proof}

\begin{theor}\label{t} Let $\mathcal{P}$ be a $\phi$-free Poisson algebra. Then $Zsoc(\mathcal{P})=N(\mathcal{P})=Ann_\mathcal{P}(Soc(\mathcal{P}))$.
\end{theor}
\begin{proof} Clearly, $Zsoc(\mathcal{P})\subseteq N(\mathcal{P})$. Let $B$ be a minimal ideal of $\mathcal{P}$ and let $N$ be a nilpotent ideal of $\mathcal{P}$. Then $B\cap N=0$ or $B$. If the former holds, then $B\cdot N+[N,B]\subseteq B\cap N=0$, so $N\subseteq Ann_\mathcal{P}(B)$. So suppose that the latter holds. Then $B\subseteq N$ and so $B\subseteq Ann_\mathcal{P}(N)$, by Lemma \ref{ann}, whence $B\cdot N+[N,B]=0$ again. It follows that $N(\mathcal{P})\subseteq Ann_\mathcal{P}(Soc(\mathcal{P}))$. It now suffices to show that $Ann_\mathcal{P}(Soc(\mathcal{P}))\subseteq Zsoc(\mathcal{P})$.
\par

There is a subalgebra $C$ of $\mathcal{P}$ such that $\mathcal{P}=Zsoc(\mathcal{P})\dot{+}C$, by Lemma \ref{splits}. Now, $Ann_\mathcal{P}(Soc(\mathcal{P}))\cap C$ is an ideal of $\mathcal{P}$ and so must contain a minimal ideal $D$. But $D$ must be a zero ideal, since it annihilates itself, by assumption. Hence $D\subseteq Zsoc(\mathcal{P})\cap C=0$ and the result is proved.
\end{proof}

\begin{theor}\label{phifree} Let ${\mathcal P}$ be a  Poisson algebra with radical $R$ and nilrdical $N$. Then ${\mathcal P}$ is $\phi$-free if and only if ${\mathcal P}=N\dot{+} U$ where $U$ is a subalgebra, $N=Zsoc({\mathcal P})$ and $R_A^2=0$. If $\mathbb{F}$ has characteristic zero, then $(U\cap R)^2=0$
\end{theor}
\begin{proof} Let ${\mathcal P}$ be $\phi$-free. Then ${\mathcal P}=Zsoc({\mathcal P})\dot{+} U$, where $U$ is a subalgebra of ${\mathcal P}$, by Theorem \ref{split}. Also, $N=N({\mathcal P})=Zsoc({\mathcal P})$, by Theorem \ref{t}. Now $(U\cap R)_A^2\subseteq R_A^2\subseteq N$, by Theorem \ref{square}, so $(U\cap R)_A^2=0$. Hence $R_A^2=N\cdot R$.  Let $Z$ be a minimal zero ideal of ${\mathcal P}$. Then $Z\cdot R$ is an ideal of ${\mathcal P}$. It follows that $Z\cdot R=Z$ or $0$. But the former is impossible, since $R$ is associative nilpotent. Thus $Z\cdot R=0$, whence $R_A^2=0$. The converse is given by Theorem \ref{split}.
\par

The characteristic zero result follows from Corollary \ref{csquare}.
\end{proof}

\begin{coro} Let ${\mathcal P}$ be a solvable Poisson algebra. Then ${\mathcal P}$ is $\phi$-free if and only if ${\mathcal P}=N\dot{+} U$ where ${\mathcal P}_A^2=0$ and ${\mathcal P}_L$ is any $\phi_L$-free solvable Lie algebra. If $F$ has characteristic zero, then $U^2=0$.
\end{coro}
\begin{proof} The fact that ${\mathcal P}=N\dot{+} U$ and ${\mathcal P}_A^2=0$ follows immediately from Theorem \ref{phifree}. Clearly, $M$ is a maximal subalgebra of ${\mathcal P}$ if and only if it is a maximal Lie subalgebra of ${\mathcal P}$, so $\phi_L({\mathcal P})=0$.
\par

The fact that $U^2=0$ in characteristic zero follows from \cite[Theorem 7.5]{frat}.
\end{proof}

\begin{theor}\label{max} Let ${\mathcal P}$ be a Poisson algebra. The the following are equivalent:
\begin{itemize}
\item[(i)] ${\mathcal P}$ is nilpotent; and
\item[(ii)] $\phi({\mathcal P})={\mathcal P}^2$.
\end{itemize} Moreover, they imply that: (iii) all maximal subalgebras of ${\mathcal P}$ are ideals.
\end{theor}
\begin{proof} $(i) \Rightarrow (ii)$: Let $M$ be a maximal subalgebra of ${\mathcal P}$. Then there exists $n$ such that ${\mathcal P}^n\not \subseteq M$ but ${\mathcal P}^{n+1}\subseteq M$. Now ${\mathcal P}=M+{\mathcal P}^n$, so
\[ {\mathcal P}^2=M\cdot {\mathcal P}+[M,{\mathcal P}]+P^{n+1}\subseteq M,
\] whence ${\mathcal P}^2\subseteq \phi({\mathcal P})$. The reverse inclusion is clear.
\par

\noindent $(ii)\Rightarrow (i)$: We have that ${\mathcal P}/\phi({\mathcal P})$ is nilpotent, whence ${\mathcal P}$ is nilpotent, by Theorem \ref{factor}.
\par

\noindent $(ii)\Rightarrow (iii)$: This is clear.
\end{proof}
\medskip

{\bf Note that} the one-dimensional Poisson algebra with basis $e$ and products $e^2=e$, $[e,e]=0$ shows that (iii) does not imply either (i) or (ii). However, algebras satisfying (iii) are not far from being nilpotent, as we will show next.

\begin{lemma}\label{lie} If all maximal subalgebras of the Poisson algebra ${\mathcal P}$ are ideals, then ${\mathcal P}$ is Lie nilpotent.
\end{lemma}
\begin{proof} Let $x \in {\mathcal P}$ and suppose that $E_{\mathcal P}^L(x)\neq {\mathcal P}$. Let $M$ be a maximal subalgebra of ${\mathcal P}$ containing  $E_{\mathcal P}^L(x)$. Then  $I_{\mathcal P}^L(M)=M$, by Lemma \ref{norm}, contradicting the fact that $M$ is an ideal of ${\mathcal P}$. It follows that $E_{\mathcal P}^L(x)= {\mathcal P}$ for all $x\in {\mathcal P}$, and so ${\mathcal P}$ is Lie nilpotent, by Engel's Theorem.  
\end{proof}

\begin{theor} All maximal subalgebras of the Poisson algebra ${\mathcal P}$ are ideals if and only if ${\mathcal P}$ is nilpotent or ${\mathcal P} = Fe\oplus N$, where $e$ is an idempotent and $N$ is the nilradical of ${\mathcal P}$.
\end{theor}
\begin{proof} Suppose first that all maximal subalgebras of ${\mathcal P}$ are ideals. If ${\mathcal P}$ is associative nilpotent, then it is nilpotent, by Lemma \ref{lie} and Proposition \ref{nil}. So suppose that it isn't nilpotent. Then it has a principal idempotent $e$. Let ${\mathcal P}= e{\mathcal P}+(1-e){\mathcal P}$ be the Peirce decomposition of ${\mathcal P}_A$. (Note that we are not assuming that there is an identity element in ${\mathcal P}$; here $(1-e){\mathcal P}=\{x-e\cdot x \mid x\in {\mathcal P}\}$.)
\par

Let $x\in {\mathcal P}$. Then
\[ [e,x]=[e\cdot e,x]=2e\cdot[e,x]=4e\cdot[e,x],
\] so $[e,x]=0$, irrespective of the characteristic of $F$. Also,
\begin{align*}
[x-e\cdot x,y-e\cdot y]& =[x,y]-[x,e\cdot y]-[e\cdot x,y]+[e\cdot x,e\cdot y] \\
& =[x,y]+e\cdot [x,y]-e\cdot [x,y]+e\cdot [x,e\cdot y] \\
&=[x,y]-e\cdot [x,y]\in (1-e){\mathcal P}.
\end{align*}
As $Fe+(1-e){\mathcal P}$ is known to be a subalgebra of ${\mathcal P}_A$, it is a subalgebra of ${\mathcal P}$.
\par

Suppose that $K=Fe+(1-e){\mathcal P}\neq {\mathcal P}$ Then $K$ is contained in a maximal subalgebra $M$ of ${\mathcal P}$. Since $M$ is an ideal of ${\mathcal P}$, we have $x=e\cdot x + x - e\cdot x\in M$ for all $x\in {\mathcal P}$, a contradiction. Hence $K={\mathcal P}$.
\par

Now $(i-e){\mathcal P}$ is associative nilpotent, since $e$ is a principal idempotent, and hence is nilpotent, by Lemma \ref{lie}. Moreover, $e\cdot (x-e\cdot x) =0$ and $[e,x-e\cdot x]=0$, as before, so $(1-e){\mathcal P}=N$ and ${\mathcal P}=Fe\oplus N$.
\par

For the converse, simply note that maximal subalgebras of $Fe\oplus N$ are either $N$ or of the form $Fe\oplus T$ where $T$ is a maximal subalgebra of $N$ and use Theorem \ref{max}.
\end{proof}


\begin{thebibliography}{1}
\bibitem{bn} D.W. Barnes and M.L. Newell, `Some Theorems on homomorphs of saturated Lie algebras', {\it Math. Z.} {\bf 115} (1970), 179-187.

\bibitem{bv} K.H. Bhaskara and K. Viswanath, ` Poisson algebras and Poisson manifolds'. Longman (1988), ISBN 0-582-01989-3.

\bibitem{cfm} M. Crainie, R.L. Fernandes and L. Marcut, `Lectures on Poisson geometry', Graduate Studies in Mathematics, {\bf 217} (2021), Amer. Math. Soc., Providence.

\bibitem{drin} V.G. Drinfield, `Quantum groups', Proceedings of the 1986 International Congress of Mathematics, {\bf 1} (1987), 798-820.

\bibitem{gr} M. Goze and E. Remm, `Poisson algebras in terms of non-associative algebras', {\it J. Algebra} {\bf 320} (2008), 294-317.

\bibitem{huch} J. Huebschmann, `Poisson cohomology and quantization', {\it J. Reine Angew. Math.} {\bf 408} (1990), 57-113.

\bibitem{kont} M. Kontsevich, `Deformatin Quantization of Poisson Manifolds', {\it Letters of Mathematical Physics} {\bf 66} (2003), 157-216.

\bibitem{lpv} C. Laurent-Gengoux, C. Pichereau and P. Vanbaecke, `Poisson Structures', Grundlehren der mathematischen Wissenschaften { \bf Vol. 347} (2013), Springer.

\bibitem{mr} M. Markl and E. Remm, `Algebras with one operation including Poisson and other Lie admissible algebras', {\it J. Algebra} {\bf 299} (2006), 171-189.

\bibitem{ont} A.  Fern\'andez Ouaridi, R.M. Navarro and D.A. Towers, `Abelian Subalgebras and Ideals of Maximal Dimension in Poisson algebras', {\it J. Algebra} {\bf 660} (2024), 680-704.

\bibitem{mr} M. Markl and E. Remm, `Algebras with one operation including Poisson and other Lie admissible algebras', {\it J. Algebra} {\bf 299} (2006), 171-189.

\bibitem{rs} V. Roubtsov and R. Such\'{a}nek, `Lectures on Poisson algebras', arxiv: 2305.03578v1 (2023).

\bibitem{su} S. Siciliano and H. Usefi, `Solvability of Poisson algebras', {\it J. Algebra} {\bf 568} (2021), 349-361.

\bibitem{frat} D.A. Towers, `A Frattini theory for algebras', {\it Proc. London Math. Soc. (3)} {\bf 27} (1973), 440-462.

\bibitem{winter} D.J. Winter, `Abstract Lie algebras', MIT. Press, Cambridge, MA, 1972

\end{thebibliography}
\end{document}